\documentclass[11pt,english]{amsart}
\usepackage[T1]{fontenc}
\usepackage[latin9]{inputenc}
\usepackage{color}
\usepackage{amsthm}
\usepackage{setspace}
\usepackage{lineno,amsmath,latexsym,amssymb,times,enumerate,hyperref}
\usepackage{verbatim}

\makeatletter
\numberwithin{equation}{section}
\numberwithin{figure}{section}
\theoremstyle{plain}
\newtheorem{thm}{\protect\theoremname}[section]
  \theoremstyle{plain}
  \newtheorem{prop}[thm]{\protect\propositionname}
  \theoremstyle{plain}
  \newtheorem{lem}[thm]{\protect\lemmaname}
  \theoremstyle{plain}
  \newtheorem{cor}[thm]{\protect\corollaryname}
  \theoremstyle{definition}
  \newtheorem{example}[thm]{\protect\examplename}

\makeatother

\usepackage{babel}
  \providecommand{\corollaryname}{Corollary}
  \providecommand{\examplename}{Example}
  \providecommand{\lemmaname}{Lemma}
  \providecommand{\propositionname}{Proposition}
\providecommand{\theoremname}{Theorem}

\begin{document}

\title{Cohen-Macaulayness of Rees Algebras of Modules}

\author{Kuei-Nuan Lin}


\begin{abstract}
We provide the sufficient conditions for Rees algebras of modules to be Cohen-Macaulay, which has been proven in the case of Rees algebras of ideals in \cite{JU} and \cite{GNN}. As it turns out the generalization from ideals to modules is not just a routine generalization, but requires
a great deal of technical development. We use the technique of generic Bourbaki
ideals introduced by Simis, Ulrich and Vasconcelos \cite{SUV} to obtain the Cohen-Macaulayness of Rees Algebras of modules.
\end{abstract}

\maketitle

\section{Introduction}

Rees algebras of modules include multi-Rees rings, which correspond
to the case where the module is a direct sum of ideals. They provide
the rings of functions on the blow up of a scheme along several subschemes.
Another important instance when Rees algebras of modules appear in
algebraic geometry is when the module is the module of differentials
or the module of top differential forms of the homogeneous coordinate
ring of a projective $k$-variety $X$, where $k$ is a field; in this case the Rees ring
tensored with $k$ is the homogeneous coordinate ring of the tangential
variety or the Gauss image of $X$ respectively \cite{SSU}. In the work of Eisenbud, Huneke, and Ulrich \cite{EHU}, they discuss
the definition of the Rees algebra of a module. Simis, Ulrich, and
Vasconcelos study the Rees algebras of modules and their arithmetical
properties in \cite{SUV}. Later in the work of Branco Correia and
Zarzuela \cite{BcZ}, they analyze the relations between the depths
of Rees algebra of a module and Rees algebra of its generic Bourbaki
ideal constructed by Simis, Ulrich, and Vasconcelos (see section 2
for the definition of generic Bourbaki ideal of a module), and use
it to deduce formula for the depth of $\mathcal{R}(E)$ for ideal
modules having small reduction number. In this work, we are interested
in finding the properties of Rees algebra of a module such as Cohen-Macaulaynees,
and of linear type, i.e. the Rees algebra of a module coincides  its
symmetric algebra. 

Let $R$ be a Noetherian local ring with residue field $k$, $E$ be a finite generated $R$-module with a rank $e>0$. We write $\mathcal{R}(E)=\mathcal{S}(E)/T_{R}$
as the Rees algebra of $E$, where 
\[\mathcal{S}(E)=R[t_{1},..,t_{n}]/(t_{1},...,t_{n})\varphi\]
is the symmetric algebra of $E$ with a representation 
\[
R^{m}\begin{array}{c}
\varphi\\
\longrightarrow\\
\\
\end{array}R^{n}\longrightarrow E\longrightarrow0,
\]
and $T_{R}$ is the $R$-torsion submodule of $\mathcal{S}(E)$. 

We write $[\mathcal{R}(E)]_{n}=E^{n}$, the $n$-th graded piece of
$\mathcal{R}(E)$ and $[\mathcal{S}(E)]_{n}=\mathcal{S}_{n}(E)$,
the $n$-the graded piece of $\mathcal{S}(E)$. If $\mathcal{R}(E)=\mathcal{S}(E)$,
we say $E$ is an module of linear type. We write $\mu(*)$ as the
minimal number of generators. Recall that we say $E$ has condition
$G_{s}$ for $s$ a positive integer, if $\mu(E_{p})\leq\mbox{dim}R_{p}+e-1$
whenever $1\leq\mbox{dim}R_{p}\leq s-1$. If $G_{s}$ is satisfied
for every $s$, $E$ is said to satisfy $G_{\infty}$. We write $\mathcal{\ell}(E)$
for the analytic spread of $E$, which is the dimension of $\mathcal{R}(E)\otimes k=\mathcal{F}(E)$,
the special fiber ring of $E$. In general, we have $\mathcal{\ell}(E)\leq d+e-1$\cite{SUV}.
Let $U$ be a submodule of $E$ and if there is an integer $r$ such that $E^{r+1}=UE^{r}$, then we say $U$ is a reduction of $E$. We define $r_{U}(E)$ to be the minimum of $r$ such that $E^{r+1}=UE^{r}$,
where $U$ is a reduction of $E$. Let $r(E)$, the reduction number
of $E$, be the minimum of $r_{U}(E)$ where $U$ ranges over all
the reduction of $E$ that are minimal with respect to inclusion.

When $E$ is an ideal $I$, one can determinate if $\mathcal{R}(I)$
is Cohen-Macaulay by checking the depth condition of $I^{n}$, analytic
spread of $I$, $l(I)$, the reduction number of $I$, $r(I)$, and
the $G_{l}$ condition by the work of Johnson and Ulrich in \cite{JU}
and of Goto, Nakamura and Nishida in \cite{GNN}. It is natural
to think if one can extend the similar results of \cite{JU} and
\cite{GNN} to the case of modules. Simis, Ulrich, and Vasconcelos
define the generic Bourbaki ideal of a module $E$, $I$, and show
the Cohen-Macaulayness of $\mathcal{R}(E)$ and $\mathcal{R}(I)$
are equivalent in \cite{SUV}. One would think it is trivial to
have the Cohen-Macaulayness of $\mathcal{R}(E)$ once $E$ has corresponding
depth conditions of $E^{n}$, $l(E)$, $r(E)$, and the $G_{l}$ condition
as the ideal case. Unfortunately, those assumptions may not be preserved
after passing through the generic Bourbaki ideal of $E$.

We start with setting notations and recalling some definitions and
theorems of the backgrounds of this work in section 2. In the section
3, we establish the sufficient conditions for a Rees algebra of a
module to be Cohen-Macaulay and a module to be of linear type in
Theorem \ref{linear}, where we use the $G_{\infty}$ condition and sliding
depth property to prove the theorem. Later we give a similar favor
of the work of Johnson and Ulrich to the module case in Theorem \ref{CM},
where it depends on if one can pass the depth condition of $E^{n}$
to its quotient (Lemma \ref{torsionFree}).

In section 4, we provide some applications of this work and are able
to extend some of work of \cite{SUV} and \cite{HSV2}(Proposition
\ref{minrank}, and Corollary \ref{CM2} and \ref{CM3}).

\section{Preliminaries}

The following notations and terminologies will be fixed through out this
work. Let $R$ be a Cohen-Macaulay local ring and write $H_{i}(\underbar{x},R)$
as the $i$-th cohomology of an $R$-ideal $I=(\underbar{x})$. If
$\mu(I)=n$ and $\mathrm{depth}H_{i}(\underbar{x};R)\geq d-n+i$ for
all $i$ where $\mbox{dim}R=d$, we say $I$ has sliding depth condition. 
In the case of an $R$-module $E$, we define $SD_{s}$ on $\mathcal{S}(E)$
as $\mathrm{depth}H_{i}(S_{+};\mathcal{S}(E))_{i}\geq d-n+i+s$ for
$i>0$, where $S_{+}$ is the irrelevant ideal of $\mathcal{S}(E)$ and $s$ is an integer.
The unqualified sliding depth condition will refer to the case $s=\mathrm{rank}(E)$.
We write $\mathcal{Z}(E)=M^{*}(S_{+};\mathcal{S}(E))$ the approximation complex
of $E$, where $M^{*}(S_{+};\mathcal{S}(E))$ is the complex defined
in \cite{HSV1}. 

We say a $R$-module $E$ with a rank has condition $\mathcal{F}_{s}$, where $s$ is an integer,
if for each prime $P$ where the localization $E_{P}$ is not a free
module then $\mu(E_{P})\leq\mathrm{rank}(E)+\mathrm{height}P-s$.
If $E$ is an ideal, $\mathcal{F}_{1}$ is the condition $G_{\infty}$.
We say $E$ is an ideal module if $E\neq0$ and torsion-free, and
the double dual $E^{**}$ is free. We call a $R$-module $E$ is reflexive
if $E^{**}=E$. A finitely generated $R$-module $E$ is orientable
if $E$ has a rank $e$ and $(\wedge^{e}E)^{**}\cong R$, where $(-)^{*}$
denotes the functor $\mbox{Hom}_{R}(-,R)$. 

We recall the Artin-Nagata property $AN_{s}$. Let $R$ be a Cohen-Macaulay
local ring, $I$ an $R$-ideal, and $s$ an integer. We say that $I$
satisfies $AN_{s}$ if $R/J:I$ is Cohen-Macaulay for every ideal
$J\subseteq I$ with $\mu(J)\leq i\leq\mathrm{height}J:I$ and every
$i\leq s$. The following proposition and theorem provide conditions
for an ideal to have Artin-Nagata property $AN_{s}$.
\begin{prop}
\label{U1-9}\cite{U} Let $R$ be a local Cohen-Macaulay ring with
dimension $d$, let $I$ be an ideal of $R$ satisfying $G_{\infty}$
and $AN_{d-2}$. Then $I$ satisfies $AN_{s}$ for any $s$ and sliding
depth.

\medskip{}
\end{prop}
\begin{thm}
\cite{H}\cite{HSV1}\label{GandSliding} Let $R$ be a local
Cohen-Macaulay ring, $I$ be an ideal of $R$ satisfying $G_{s}$
and sliding depth. Then $I$ satisfies $AN_{s}$.
\end{thm}
We recall the definition of the generic Bourbaki ideal of a module
$E$. We write $R(Z)=R[Z]_{mR[Z]},$ where $Z$ is a set of indeterminates
over $R$. Let $U=\sum_{i=1}^{n}Ra_{i}\subset E$ and $E/U$ a torsion
$R$-module. Let $z_{ij}$, $1\leq i\leq n$, $1\leq j\leq e-1$ be
indeterminates and $R''=R(\{z_{ij}\})$, $U''=R''\otimes U$, $E''=R''\otimes E$,
$x_{j}=\sum_{i=1}^{n}z_{ij}a_{i}\in E''$, $F=\sum_{j=1}^{e-1}Rx_{j}$
which is free $R''$-module of rank $e-1$. Assume $E''/F$ is $R''$-torsion
free then $E''/F\simeq I$ for some $R''$-ideal $I$ with grade$I>0$.
$I$ is called a generic Bourbaki ideal of $E$ with respect to $U$.
If $U=E$, we write $I=I(E)$, the generic Bourbaki ideal of $E$.
Note: $I$ is not unique. The following theorem and lemma show that
some properties of the Rees algebra of a module $E$ and the Rees algebra
of its generic Bourbaki ideal are equivalent.
\begin{thm}
\label{EtoI}\cite{SUV} Let $(R,m)$ be a Noetherian local ring,
$E$ a finite generated $R$-module with $\mathrm{rank}E=e>0$. Let
$I\simeq E''/F$ be a generic Bourbaki ideal of $E$. 
\begin{enumerate}
\item
\begin{enumerate}
\item
$\mathcal{R}(E)$ is Cohen-Macaulay if and only if $\mathcal{R}(I)$
is Cohen-Macaulay.
\item 
$E$ is of linear type with $\mathrm{grade}\mathcal{R}(E)_{+}\geq e$
if and only if $I$ is of linear type.
\end{enumerate}
\item If any of conditions (a), (b) of (1) hold, then $\mathcal{R}(E'')/F\cong\mathcal{R}(I)$
and $x_{1},...,x_{e-1}$ form a regular sequence on $\mathcal{R}(E'')$.
\end{enumerate}
\end{thm}
\begin{lem}
\label{G_l}\cite{SUV} Let $(R,m)$ be a Noetherian local ring,
$E$ be a finite generated $R$-module with $\mathrm{rank}E=e\geq2$.
$E$ satisfies $G_{l}$. Then (a)
$\mathcal{\ell}(\overline{E})=\mathcal{\mathcal{\ell}}(E)-1$, (b)
$r(\bar{E})\leq r(E)$ (c) $\bar{E}$ satisfies $G_{l}$, where $\overline{E}=E''/(x)$.\end{lem}
\begin{prop}
\cite{SUV}\label{S-U-V 5.3} Let $R$ be a Cohen-Macaulay local ring
and $E$ be an ideal module. Let $U$ be a submodule of E. Assume that $\mathrm{grade}E/U\geq s\geq2$
and that $E$ is free locally in codimension $s-1$. Further let $I$ be
a generic Bourbaki ideal of $E$ with respect to $U$. Then $I$ satisfies
$G_{s}$ and $AN_{s-1}$.
\end{prop}
Our main motivation of this work is to provide the similar results
of following theorems to the case of modules.
\begin{thm}
\label{J-U0}\cite{JU} Let $R$ be a Cohen-Macaulay local ring of
dimension $d$ with infinite residue field, $I$ be an $R$-ideal
with height $g\geq2$, $\mathcal{\ell}(I)=l$, $r(I)=r$. Let $s$
be an integer with $r\leq s$. Assume $I$ satisfies $G_{l}$ and
$AN_{l-3}$ locally in codimension $l-1$, that $I$ satisfies $AN_{l-\mbox{max}\{2,s\}}$,
and that $\mathrm{depth}\, I^{n}\geq d-l+s-n+1$ for $1\leq n\leq s.$
Then $\mathcal{R}(I)$ is Cohen-Macaulay.
\end{thm}
\medskip{}

When $s=l-g+1$ and $I$ is $G_l$, the Artin-Nagata property is automatically satisfied.
\begin{cor}
\label{J-U3-1}\cite{JU} Let $R$ be a Cohen-Macaulay local ring of
dimension $d$ with infinite residue field, $I$ be an $R$-ideal
with height $g\geq2$, $\mathcal{\ell}(I)=l$, $r(I)\leq l-g+1$.
Assume $I$ is $G_{l}$ and $\mathrm{depth}\, I^{n}\geq d-g-n+2$
for $1\leq n\leq l-g+1.$ Then $\mathcal{R}(I)$ is Cohen-Macaulay.
\end{cor}
\medskip{}

\begin{thm}
\cite{JU}\label{JU2} Let $(R,m)$ be a Gorenstein local ring of
dimension $d$ with infinite residue field. Let $I$ be an $R$-ideal
with $\mathrm{grade}I=g\geq2$ and $\ell(I)=l$, $r(I)=r\leq l-g$.
If $I$ has $G_{l}$ condition and $\mathrm{depth}I^{n}\geq d-n-g+2$
for $1\leq n\leq l-g$. Then $\mathcal{R}(I)$ is Gorenstein.
\end{thm}
In order to understand when $\mathcal{R}(E)$ is Cohen-Macaulay with
respect to the depth bound of $E^{n}$ and $G_{\infty}$ condition
similar to above theorems, we need the following theorem that gives
sufficient conditions for an ideal to have sliding depth and $G_{\infty}$.
Then we want to pass the conditions on a module to its generic generic
Bourbaki ideal.
\begin{thm}
\label{U0}\cite{U} Let $R$ be a Gorenstein local ring with dimension
$d$, $I$ be an $R$-ideal with $\mathrm{grade}I=g$. We assume $I$
has condition $G_{\infty}$ and $\mathcal{\ell}(I)=l$. If $\mathrm{depth}I^{n}\geq d-g-n+2$
for $1\leq n\leq l-g+1$, then $I$ is sliding depth.
\end{thm}
\medskip{}

The acyclicity of the $\mathcal{Z}(E)$ complex could guarantee the
existence of the Bourbaki ideal of a module such that the ideal is
generated by a proper sequence. The following theorem combining with
Theorem \ref{EtoI} give us equivalent conditions between the module
and the generic Bourbaki ideal of the module once one can show the
$\mathcal{Z}(E)$ complex is acyclic.
\begin{thm}
\label{H-S-V5_6}\cite{HSV2} Let $R$ be a local ring with infinite
residue field. Suppose E is a finitely generated torsion-free R-module
and either (i) $\mathrm{pd}E<\infty$ or (ii) $R$ is a normal domain.
The following conditions are equivalent: (a) $\mathcal{Z}(E)$
is acyclic. (b) $E$ admits a Bourbaki sequence $0\rightarrow I\rightarrow F\rightarrow E\rightarrow0$
such that: (b1) $F$ is generated by elements which form a regular
sequence of 1-forms on $\mathcal{S}(E)$. (b2) $I$ is generated by
a proper sequence.
\end{thm}
By checking the sliding depth and $\mathcal{F}_{0}$ condition, we
can see that if the symmetric algebra of a module is Cohen-Macaulay.
In particular, if one can show that a module is of linear type and
the module has the sliding depth and $\mathcal{F}_{0}$ condition,
the Cohen-Macaulayness of the Rees algebra of the module is provided. 
\begin{thm}
\label{HSV6.2}\cite{HSV2} Let $R$ be a Cohen-Macaulay local
ring and $E$ a finitely generated R-module of rank$(E)=e$. The following
conditions are equivalent: (a) $\mathcal{Z}(E)$ is acyclic and $\mathcal{S}(E)$
is Cohen-Macaulay. (b) $E$ satisfies sliding depth $SD_{e}$ and
$\mathcal{F}_{0}$.
\end{thm}
The following corollary gives the Cohen-Macaulayness of $\mathcal{S}(I)$
and $I$ is of linear type once one has the depth bound of $I^{n}$
by using Theorem \ref{U0}.
\begin{cor}
\label{HSV6.3}\cite{HSV2} Let $R$ be a Cohen-Macaulay ring and
$E$ be a $R$-module with sliding depth and $G_{\infty}$ conditions.
Then $\mathcal{S}(E)$ is Cohen-Macaulay and $E$ is an ideal of linear
type.
\end{cor}

\section{Rees Algebras of modules}

We are now ready to state and prove the main theorem. 
\begin{thm}
\label{linear}Let $R$ be a local Gorenstein ring of dimension $d$,
and $E$ be a finite generated orientable module with condition $G_{\infty}$
and $\mbox{rank}E=e$. Let $I\simeq E''/F$ be a generic Bourbaki ideal
of $E$. If $\mbox{grade}I\geq2$ and $\mathrm{depth}E^{n}\geq d-n$
for $1\leq n\leq l-e$ where $l=\mathcal{\ell}(E)$, then $E$ is
of linear type and $\mathcal{R}(E)$ is Cohen-Macaulay.\end{thm}
\begin{proof}
We induct on dimension of $R$. Suppose $E_{p}$ is linear on $\mbox{dim}R_{p}\leq d-1$
and $\mathcal{R}(E_{p})$ is Cohen- Macaulay. Let $I\simeq E''/F$
be a generic Bourbaki ideal of $E$. We would like to apply Theorem
\ref{EtoI}, hence we will show $\mbox{depth}I^{n}\geq d-n$. Without
lost of generality, we assume $\mbox{rank}E=2$, hence it is sufficient
to show $\mbox{depth}(E''/x)^{n}\geq d-n$ for $1\leq n\leq l-2$, where $x$ is a regular element of $\mathcal{R}(E'')$.

Claim: $(E'')^{n}/x(E'')^{n-1}=(E^{''}/x)^{n}$, i.e. $(E'')^{n}/x(E'')^{n-1}$
is torsion-free for $1\leq n\leq l-2$.

With the claim and the exact sequence
\begin{equation}
0\rightarrow x(E'')^{n-1}\rightarrow(E'')^{n}\rightarrow(E'')^{n}/x(E'')^{n-1}\rightarrow0,\label{xE}
\end{equation}
we apply the depth condition of $E''$ which is as the depth condition
of $E$. This gives us $\mbox{depth}I^{n}=\mbox{depth}(E''/x)^{n}=\mbox{depth}(E'')^{n}/x(E'')^{n-1}\geq d-n$
for $1\leq n\leq l-2$. By Lemma \ref{G_l} and Theorem \ref{U0},
$I$ is sliding depth and satisfies $G_{\infty}$. By the Corollary
\ref{HSV6.3}, $I$ is of linear type and $\mathcal{R}(I)$ is Cohen-Macaulay. Then by Theorem \ref{EtoI}, $E$ is of linear type and $\mathcal{R}(E)$
is Cohen-Macaulay.

Proof of claim: Notice $(E'')^{n}/x(E'')^{n-1}$ is torsion-free if
and only if for any $p\in\mbox{Spec}R$, we have $\mbox{depth}((E'')^{n}/x(E'')^{n-1})_{p}\geq1$.
Since $E$ is of linear type on codimension $d-1$, $E''$ is
also of linear type on codimension $d-1$. We have 
\[
(E'')^{n}/x(E'')^{n-1}=\mathcal{S}_{n}(E'')/x\mathcal{S}_{n-1}(E'')=(E''/xE'')^{n}=[\mathcal{R}(E''/xE'')]_{n},
\]
hence $(E'')^{n}/x(E'')^{n-1}$ is torsion-free on codimension $d-1$.
Since we assume $\mbox{rank}E=2,$ $\mathcal{\ell}(E)=l\leq d+2-1=d+1$ and depth condition of $E$ passing through $E''$. By using the exact
sequence \ref{xE}, $\mbox{depth}(E'')^{n}/x(E'')^{n-1}\geq d-n\geq d-(l-2)\geq1$.
This implies $(E'')^{n}/x(E'')^{n-1}$ is torsion-free.
\end{proof}
The following example is computed in Macaulay 2 \cite{GS}.
\begin{example}
$R=k[x,y,z]_{(x,y,z)}$ where $k$ is a field. Let $E=(x^{2},xy)\oplus(y,z)$
be a module of rank 2 satisfying $G_{\infty}$ condition with $l(E)=4$.
The generic Bourbaki ideal of $E$ is $(z_{11}x^{4}+z_{21}x^{3}y,\mbox{ }z_{11}x^{3}y+z_{21}x^{2}y^{2},\mbox{ }z_{31}y^{2}+z_{41}yz,\mbox{ }z_{31}yz+z_{41}z^{2})$
in $R(z_{11},z_{21},z_{31},z_{41})$ with grade at least 2. The depth
conditions that we need to check are when $n=1$ and $n=2$. We have
depth $E^{1}=2\geq3-1$ and depth$E^{2}=2\geq3-2$. Hence $E$ is
of linear type and is Cohen-Macaulay.
\end{example}
We would like to weaken the depth bound on $E^{n}$ and $G_{\infty}$
condition to more general modules. One key point that one can
pass from a module $E$ to its generic Bourbaki ideal is to show the
$(E'')^{n}$ is still torsion-free after modulo a regular sequence of
$\mathcal{R}(E'')$. The following lemma shows when the torsion-freeness
will be provided.
\begin{lem}
\label{torsionFree}Let $R$ be a Gorenstein local ring. Let $E$
be a finitely generated torsion free $R$-module with rank $e>0$
and $E$ is orientable and $E$ is $G_{l-e+1}$. Let $I$ be a generic
Bourbaki ideal of $E$ having height $g=2$ and $\mathrm{depth}\, E^{n}\geq d-n$
for $1\leq n\leq l-e$. We write $\overline{E}=E''/(x)$. Then $\overline{E}$
is an ideal of linear type on codimension $l-e$ and $\overline{E}^{n}=(E'')^{n}/x(E'')^{n-1}$.\end{lem}
\begin{proof}
Theorem \ref{linear} gives $E''$ is of linear type on codimension
$l-e$. Then by Theorem \ref{EtoI}, we have $\bar{E}$ is also of
linear type on codimension $l-e$. Hence $\bar{E}^{n}=(E'')^{n}/x(E'')^{n-1}$
on codimension $l-e$. For prime ideal $p$ with $\mbox{height}(p)>l-e$,
we use depth conditions of $E$ which can be passed to $E''$ and
the exact sequence $0\rightarrow(x(E'')^{n-1})_{p}\rightarrow((E'')^{n})_{p}\rightarrow((E'')^{n}/x(E'')^{n-1})_{p}\rightarrow0$
to show $\mbox{depth}((E'')^{n}/x(E'')^{n-1})_{p}\geq d-n\geq d-(l-e)\geq1$
and this gives $((E'')^{n}/x(E'')^{n-1})_{p}$ is $R$-torsion free on
$R_{p}$. We obtain $(E'')^{n}/x(E'')^{n-1}$ is $R$-torsion free and
hence $\bar{E}^{n}=(E'')^{n}/x(E''){}^{n-1}$.
\end{proof}
We have the result of the module case similar to the ideal case in
the following theorem. 
\begin{thm}
\label{CM}Let $R$ be a Gorenstein local ring of dimension $d$ with
infinite residue field. Let $E$ be a finitely generated torsion free
$R$-module with rank $e>0$ and $\mathcal{\ell}(E)=l$. Assume $E$
has condition $G_{l-e+1}$ with $l-e+1\geq2$, and $E$ is orientable.
Let $I$ be a generic Bourbaki ideal of $E$ having height $g\geq2$.
If $r(E)\leq l-g-e+1$, and $\mathrm{depth}\, E^{n}\geq d-g-n+2\mbox{ for }1\leq n\leq l-e-g+1$,
then $\mathcal{R}(E)$ is Cohen-Macaulay.\end{thm}
\begin{proof} For $g\geq3$, we have $E\cong F\oplus L$ where $F$ is
a free $R$-module of rank $e-1$ and $L$ is an $R$-ideal \cite{SUV}.
We have $\mathcal{R}(E)=\mathcal{R}(L)[t_{1},...,t_{e-1}]$ where
$t_{i}$ are variables. Then $\mathcal{R}(L)$ is Cohen-Macaulay if and only if
$\mathcal{R}(E)$ is Cohen-Macaulay. But all the conditions of $E$ pass through
$L$ and by Corollary \ref{J-U3-1}, we have $\mathcal{R}(L)$ is Cohen-Macaulay.
For the case $g=2$, we use induction on $e$ to show $\mathcal{R}(I)$
is Gorenstein. Then by Theorem \ref{EtoI}, we have $\mathcal{R}(E)$
is Cohen-Macaulay.

By Lemma \ref{torsionFree}, we have $\bar{E}^{n}=(E'')^{n}/x(E'')^{n-1}$,
then we have $\mathrm{depth}\bar{E}^{n}\geq d-n$ by the exact sequence
$0\rightarrow x(E'')^{n-1}\rightarrow (E'')^{n}\rightarrow (E'')^{n}/x(E'')^{n-1}\rightarrow0$
and the depth condition of $E$ passing through $E''$. Then by induction
on $e$, we have $\mathrm{depth}(E''/F)^{n}=I^{n}\geq d-n$. By Lemma
\ref{G_l}, $I$ satisfies $G_{l-e+1}$ with $l(I)=l-e+1$ and $r(I)\leq l-e+1-g$,
then by Theorem \ref{JU2}, we have $\mathcal{R}(I)$ is Gorenstein.
\end{proof}

\section{Application}

We are able to check if a module is of linear type and Cohen-Macaulay
if the minimal number of the generators of the module is small enough
with respect to the dimension of the ground ring $R$ and the rank
of the module, and the depth of the module is large enough. 
\begin{prop}
\label{minrank}Let $R$ be a Cohen-Macaulay local ring with dimension
$d$, $E$ be finitely generated $R$-module with $\mathrm{rank}(E)=e$.
Assume $E$ satisfies $G_{\infty}$, $\mu(E)\leq\mbox{min}\{e+3,d+e-1\}$,
and $\mathrm{depth}E\geq d-1$, then $E$ is of linear type and $\mathcal{R}(E)$
is Cohen-Macaulay.\end{prop}
\begin{proof}
Let $0\rightarrow U\rightarrow F\rightarrow E\rightarrow0$ be an
exact sequence where $F$ is a free $R$-module with rank $\mu(E)$.
Then $U$ has rank $\mu(E)-e\leq 3$ and $\mbox{depth}U=d$. We may assume $\mbox{rank}U>0$. We would
like to show the $\mathcal{Z}(E)$ complex is acyclic when $0<$rank$U\leq3$.

When $\mbox{rank}U=1=\mu(E)-e$, we have $\mbox{dim}R=d\geq2$. Consider
the following sequence $0\rightarrow U\otimes \mathcal{S}(F)\rightarrow \mathcal{S}(F)\rightarrow \mathcal{S}(E)\rightarrow0$.
Notice that $\mathcal{S}(F)$ is a polynomial ring with $\mu(E)$
variables over the ring $R$ which is Cohen-Macaulay and $U\otimes\mathcal{S}(F)$
is a module over $\mathcal{S}(F)$, hence we have $\mbox{depth}U\otimes\mathcal{S}(F)=d\geq2$
and $\mbox{depth}\mathcal{S}(F)=d\geq1$. Therefore the sequence is
exact. The depth condition implies $E$ is sliding depth. With $G_{\infty}$,
we can apply Corollary \ref{HSV6.3}, hence $\mathcal{S}(E)$ is Cohen-Macaulay
and $E$ is of linear type, i.e. $\mathcal{R}(E)$ is Cohen-Macaulay.

When $\mbox{rank}U=2=\mu(E)-e$, we have $\mbox{depth}U=d\geq3$ and
$(\wedge^{2}U)^{**}\cong R$. Consider the following sequence $0\rightarrow(\wedge^{2}U)^{**}\otimes\mathcal{S}(F)\rightarrow U\otimes\mathcal{S}(F)\rightarrow\mathcal{S}(F)\rightarrow\mathcal{S}(E)\rightarrow0$.
As before, we have the depth conditions on this sequence and hence
it is exact. Therefore $\mathcal{S}(E)$ is Cohen-Macaulay and $E$
is of linear type, i.e. $\mathcal{R}(E)$ is Cohen-Macaulay. 

When $\mbox{rank}U=3$, we have $\mbox{depth}U=d\geq4$, $(\wedge^{3}U)^{**}\cong R$,
and $(\wedge^{2}U)^{**}\cong U^{*}$. Now use similar approach as
above. 
\end{proof}
\medskip{}

\begin{example}
Let $R=[x,y,z]_{(x,y,z)}$ and $E=(x^{2},xy)\oplus(yz,z^{2})$, then
$\mu(E)=4$ with rank$E=2$, and $G_{\infty}$ is true. The generic
Bourbaki ideal of $E$, $I=((z_{11}x^{2}+z_{21}xy)(x^{2},xy),(z_{31}yz+z_{41}z^{2})(yz,z^{2}))$
has height 2, and depth$E^{1}=2\geq3-1$. Then by the Proposition
\ref{minrank}, $E$ is of linear type and $\mathcal{R}(E)$ is Cohen-Macaulay. 
\end{example}
The following corollaries are applications of Proposition \ref{S-U-V 5.3}
and Theorem \ref{H-S-V5_6}. We can weaken the depth assumptions if
the module is free locally on certain dimension. 
\begin{cor}
\label{CM2}Let $R$ be a Cohen-Macaulay local ring, $E$ be an
ideal module with $\ell(E)=l$ and $\mathrm{rank}E=e$. Let $U$ be
a submodule of $E$ such that $\mbox{grade}E/U=l-e-1\geq2$. Assume
$E$ has condition $G_{l-e+1}$ and $E$ is free on codimension $l-e-2$.
If $1\leq r(E)\leq s$ for some integer $s$ and $\mathrm{depth}E^{n}\geq d-(l-e+1)+s-n+1$
for $1\leq n\leq s$, then $\mathcal{R}(E)$ is Cohen-Macaulay.\end{cor}
\begin{proof}
Let $I$ be the generic Bourbaki ideal of $E$ with respect to $U$.
We will show $(E'')^{n}/x(E'')^{n}$ is torsion-free as in section
3. Then we can pass all the depth conditions to the ideal $I$, hence
by Lemma \ref{G_l}, Theorem \ref{J-U0} and \ref{EtoI}, we have
$\mathcal{R}(E)$ is Cohen-Macaulay. 

On codimension $l-e$, by Lemma \ref{G_l} and Proposition \ref{S-U-V 5.3},
we have $I$ is $G_{l-e+1}$ and $AN_{l-e-2}$. Hence on codimension
$l-e$, $I$ is $G_{\infty}$ and $AN_{l-e-2}$. Then by Proposition
\ref{U1-9}, $I$ satisfies sliding depth. Hence by Corollary \ref{HSV6.3},
$I$ is an ideal of linear type and $E$ is an module of linear type
by Theorem \ref{EtoI} on codimension $l-e$. Therefore $(E'')^{n}/x(E'')^{n}$
is torsion-free on codimension $l-e$.

On codimension $d>l-e$, we apply depth conditions on the exact sequence
$0\rightarrow x(E'')^{n}\rightarrow(E'')^{n}\rightarrow(E'')^{n}/x(E'')^{n}\rightarrow0$
and use the facts that $d-(l-e)\geq1$ and $s\geq n$, then we obtain
$\mbox{depth}(E'')^{n}/x(E'')^{n}\geq1$ and $(E'')^{n}/x(E'')^{n}$
is torsion-free.
\end{proof}
We are able to loosen the locally free condition if we have some control
on the projective dimension of the module locally.
\begin{cor}
\label{CM3}Let $R$ be a Cohen-Macaulay local ring, and $E$ be an
ideal module with $\ell(E)=l$, $\mathrm{rank}E=e$ and $r(E)=r\geq1$.
Assume $E$ has condition $G_{l-e+1}$. Let $s$ be an integer such
that $E$ is free on codimension $l-e-s$ and $r\leq s$. If $\mathrm{Projdim}E_{P}=2$
and $\mu(E_{P})\leq1/2(\mathrm{height}(P)-1)+e$ for all $\mathrm{height}P\leq l-e$,
and $\mathrm{\mathrm{depth}}E^{n}\geq d-(l-e+1)+s-n+1$ for all $1\leq n\leq s$, then $\mathcal{R}(E)$ is Cohen-Macaulay.\end{cor}
\begin{proof}
Let $I$ be a generic Bourbaki ideal of $E$. Notice that $I$ satisfies
$AN_{l-e-s}$ and $G_{l-e+1}$ with $r(I)\leq s$ and $l(I)=l-e+1$
by Lemma \ref{G_l} and Proposition \ref{S-U-V 5.3}. As before, we
will show the depth conditions can pass to the ideal $I$, i.e. we
will show $(E'')^{n}/x(E'')^{n}$ is torsion-free. 

Notice that on codimension $d>l-e$, the depth condition is the same
as in Corollary \ref{CM2}, hence $(E'')^{n}/x(E'')^{n}$ is torsion-free
on codimension $d>l-e$. On codimension $d\leq l-e$, we consider
the exact sequence $0\rightarrow L\rightarrow R^{n}\rightarrow E\rightarrow0$, then $\mbox{rank}L_{P}=\mu(E_{P})-e\leq1/2(\mbox{height}(P)-1)$ for all
$\mbox{height}P\leq l-e$. Since $\mbox{Projdim}E_{P}=2$, $\mbox{Projdim}\wedge^{t}L_{P}=t$
for all $1\leq t\leq\mbox{rank}L_{P}$, hence $\mbox{depth}\wedge^{t}L_{P}=\mbox{height}P-t\geq t+1\geq2$
for all $\mbox{height}P\leq l-e$ and $\wedge^{t}L_{P}$ is reflexive.
We obtain the $\mathcal{Z}(E)$-complex is acyclic and $E$ is of
linear type on codimension $l-e$ and therefore $I$ is of linear
type on codimension $l-e$ by Theorem \ref{EtoI}. By using the same
proof of Lemma \ref{torsionFree}, we have $(E'')^{n}/x(E'')^{n}$
is torsion-free.

To prove $\mathcal{R}(E)$ is Cohen-Macaulay, we use induction on
$d$ and assume $\mathcal{R}(E)$ is Cohen-Macaulay on codimension
$l-e$. Since $E$ is of linear type and $\mathcal{Z}(E)$-complex
is acyclic on codimension $l-e$, then by Theorem \ref{H-S-V5_6}
and Theorem \ref{EtoI}, $\mathcal{Z}(I)$-complex is acyclic and
$\mathcal{S}(I)=\mathcal{R}(I)$ is Cohen-Macaulay hence $I$ is sliding
depth on codimension $l-e$ by Theorem \ref{HSV6.2}. Moreover, since
$I$ is of linear type, $I$ satisfies $G_{\infty}$. Therefore
$I$ is $AN_{\infty}$ on codimension $l-e$ by Theorem \ref{GandSliding}.
This shows $\mathcal{R}(I)$ is Cohen-Macaulay by Theorem \ref{J-U0},
hence by Theorem \ref{EtoI}, $\mathcal{R}(E)$ is Cohen-Macaulay. 
\end{proof}
\textbf{Acknowledgments:} Part of this work was done when the author
was a graduate student at Purdue University under the direction of
Professor Bernd Ulrich. The author is very grateful for so many useful
suggestions from Professor Ulrich.

\address{\begin{center}
{\small{}The Penn State University, Greater Allegheny Campus, 4000
University Dr., McKeesport, PA 15132, USA}
\par\end{center}}

\email{\begin{center}
kul20@psu.edu 

http://www.personal.psu.edu/kul20/
\par\end{center}}


\begin{thebibliography}{99}
\bibitem{AN} M. Artin and M. Nagata, Residual intersections
in Cohen-Macaulay rings. J. Math. Kyoto Univ. \textbf{12 }(1972),
307-323.

\bibitem{BcZ} A. Branco Correia and Zarzuela, On the depth
of the Rees algebra of an ideal module. J. Algebra 323 (2010), no.
5, 1503\textendash 1529. 

\bibitem{EHU} D. Eisenbud, C. Huneke,and B. Ulrich, What is
the Rees algebra of a module? Proc. Amer. Math. Soc. 131 (3) (2002)
701\textendash{} 708.

\bibitem{GNN} S. Goto, Y. Nakamura, and K. Nishida, Cohen-Macaulay
graded rings associated to ideals. Amer. J. Math. \textbf{118} (1996),
1197-1213.

\bibitem{GS} D. Grayson and M. Stillman, Macaulay2, a software
system for research in algebraic geometry, Available at http://www.math.uiuc.edu/Macaulay2/.

\bibitem{HSV1} J. Herzog, A. Simis, and W. V. Vasconcelos, Koszul
homology and blowing-up rings, Proc. Trento Commutative Algebra Conf.,
Lectures Notes in Pureand AppliedMath.,vol.84,Dekker,NewYork, 1983,
79-169.

\bibitem{HSV2} J. Herzog, A. Simis, and W. V. Vasconcelos. On
the arithmetic and homology of algebras of linear type. Trans. Amer.
Math. Soc. \textbf{283} (1984), 661\textendash 683.

\bibitem{HVV} J. Herzog, W. V. Vasconcelos, and R.H. Villarreal.
Ideals with sliding depth. Nagoya Math. J. \textbf{V. 99} (1985),
159-172

\bibitem{H} C. Huneke, Strongly Cohen-Macaulay schemes and residual
intersections. Trans. Amer. Math. Soc. \textbf{277} (1983), 739\textendash 763.

\bibitem{JU} M. Johnson and B. Ulrich, Artin-Nagata properties
and Cohen-Macaulay associated graded rings, Composition Math. 103
(1996), 7-29. 

\bibitem{SSU}A. Simis, K. Smith, and B. Ulrich, An algebraic
proof of Zak's inequality for the dimension of the Gauss image, Math.
Z. \textbf{241} (2002), 871-881. 

\bibitem{SUV}A. Simis, B. Ulrich, and W. V. Vasconcelos, Rees
algebras of modules, Proc. London Math. Soc. \textbf{87} (2003), 610-646.

\bibitem{U}B. Ulrich, Artin-Nagata properties and reductions of
ideals. Contemp. Math. \textbf{159}, Amer. Math. Soc., Providence,
RI, 1994.

\end{thebibliography}
\end{document}